\documentclass[12pt]{article}
\usepackage{mathrsfs}
\usepackage{amsmath}
\usepackage{float}
\usepackage{cite}
\usepackage{esint}
\usepackage[all]{xy}
\usepackage{amssymb}
\usepackage{amsfonts}
\usepackage{amsthm}
\usepackage{color}
\usepackage{graphicx}
\usepackage{hyperref}
\usepackage{bm}
\usepackage{indentfirst}
\usepackage{geometry}
\theoremstyle{plain}%plain,definition,remark. need amsthm package
\newtheorem{thm}{Theorem}[section]

\newtheorem{defn}[thm]{Definition}

\newtheorem{cor}[thm]{Corollary}
\newtheorem{lem}[thm]{Lemma}

\newtheorem{them}[thm]{Theorem}
\theoremstyle{remark}
\newtheorem{ex}[thm]{Example}
\newtheorem{rmk}[thm]{Remark}

\newcommand{\R}{\mathbb{R}}

\newcommand{\C}{\mathbb{C}}

\newcommand{\CF}{\mathrm{CF}}

\newcommand{\Crit}{\mathrm{Crit}}
\newcommand{\D}{\mathbb{D}}


\title{Equivalence of the pearly tree immersed Lagrangian Floer theory and the Hamiltonian immersed Lagrangian Floer theory}
\author{Zuyi Zhang }
\date{ }

\begin{document}

\maketitle

\begin{abstract}
    The goal of this paper is to prove an equivalence relation between the immersed Lagrangian Floer theory, defined using pearly tree discs, and local Hamiltonian flows, i.e., Hamiltonian flows performed in the Weinstein tubular neighborhood. This is a generalization of Alston-Bao's work.
\end{abstract}

\section{Introduction}
The goal of this paper is to prove an equivalence relation between the immersed Lagrangian Floer theory, defined using pearly tree discs, and local Hamiltonian flows, i.e., Hamiltonian flows performed in the Weinstein tubular neighborhood.\\

The Lagrangian Floer theory was first introduced by Floer to solve Arnold's conjecture in \cite{floer1988morse}. Given an embedded Lagrangian submanifold $L$ in a symplectic manifold $M$, Floer constructed the Lagrangian Floer homology whose generators are the intersections $L\cap\phi(L)$, and the boundary map is definde by counting the numbers of holomorphic discs connecting the points in $L\cap\phi(L)$. Here, $\phi$ is a Hamiltonian diffeomorphism on $M$ such that $L$ and $\phi(L)$ intersects transversely. The Hamiltonian immersed Lagrangian Floer theory generalizes Floer's work to immersed Lagrangian submanifold. elated works include those by Akaho-Joyce \cite{akaho2010immersed}, Fukaya-Oh-Ohta-Ono \cite{fukaya2009lagrangian}, and Abouzaid \cite{abouzaid2008fukaya}, among others.\\

The pearly tree version of Lagrangian Floer theory was first introduced by Oh \cite{oh1994relative} and later studied extensively by Biran-Cornea \cite{biran2007quantum}, \cite{biran2009lagrangian}. Other related works include those by Woodward-Palmer \cite{palmer2021invariance}. Given a morse function $f$ on a Lagrangian submanifold $L$ in a symplectic manfold $M$, the generators of Lagrangian Floer theory in this version are the critical points of $f$. The boundary map is defined by counting the number of pearly trajectories whose moduli space has dimension 0. The pearly trajectories are by definition to be the holomorphic discs connected by the gradient flows of $f$ (Figure \ref{fig:Pearlytraj}). In \cite{alston2021immersed}, when the Lagrangian submanifold is an immersion and the Lagrangian Floer chain group is a chain complex, Alston and Bao proved that when the Lagrangian submanifold is an immersion and the Lagrangian Floer chain group forms a chain complex, the Lagrangian Floer homology in this version is isomorphic to the Lagrangian Floer homology using Hamiltonian flows, assuming additional conditions on the Maslov index.\\

In this paper, we generalize the work of Alston-Bao \cite{alston2021immersed} to the case where the Lagrangian Floer chain group is not necessarily a chain complex (commonly referred to as obstructed). Instead of proving a quasi-isomorphism as in Alston-Bao, we establish a canonical identification between the generators and boundary maps of these two chain groups. The main theorem of this paper is as follows:

\begin{them}
    Let $(M^n,\omega)$ be a closed symplectic manifold. Suppose $g:L\looparrowright M$ is a Lagrangian immersion. Assume $f$ is a morse function defined on $L$ whose critical points are distinct from the self-intersections of $L$. Given a regular almost complex structure on $M$,  there is a local Hamiltonian flow $\phi_t$ such that the pearly tree immersed Lagrangian Floer chain group $(\CF^P(L),\partial^P)$, defined by $f$, is identified with the Hamiltonain immersed Lagrangian Floer chain group $(\CF^H(L,L_{\phi_\epsilon}),\partial^H)$, where $L_{\phi_\epsilon}$ denotes the Lagrangian immersion $\phi_\epsilon\circ g$ for some small $\epsilon>0$.
    
    Conversely, equipped with a regular almost complex structure on $M$, given a local Hamiltonian flow $\phi_t$ such that the corresponding Hamiltonian function restricted to $L$ is a morse function $f$, there is a canonical identification between $(\CF^H(L,L_{\phi_\epsilon}),\partial^H)$ and the pearly tree immersed Lagrangian Floer chain group $(\CF^P(L),\partial^P)$, defined by $f$, for sufficiently small $\epsilon>0$.
\end{them}

\begin{figure}[H]
\centering 
\includegraphics[width=0.2\textwidth]{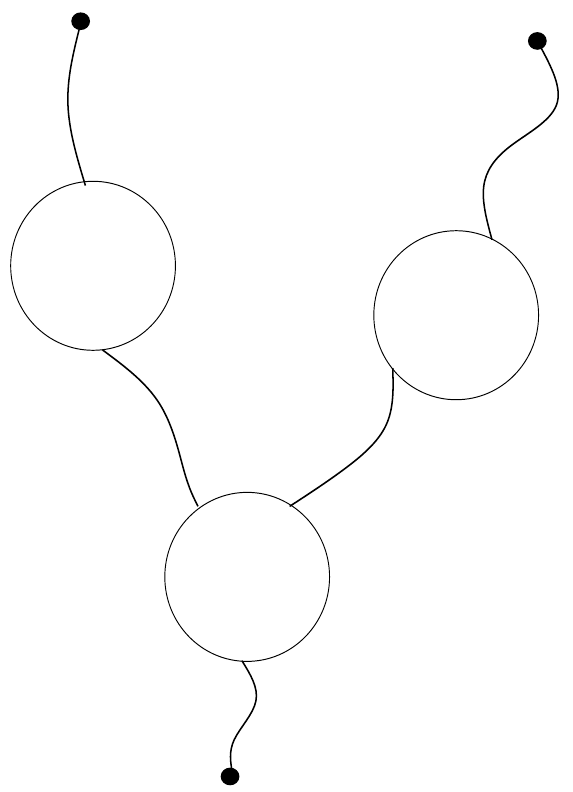}
\caption{The curves are gradient flows and the discs are holomorphic discs. The top dots are inputs and the bottom dot is the output.}
\label{fig:Pearlytraj}
\end{figure}

\section{Preliminary}
The aim of this section is to provide the necessary background on the moduli space of $J$-holomorphic discs with Lagrangian boundary conditions. The second part of this section presents some transversality results for the moduli space.

\begin{defn}
    Let $(M,\omega)$ be a symplectic manifold. \textbf{An almost complex structure} $J$ on $M$ is a bundle map such that
    \begin{itemize}
        \item $J:TM\rightarrow TM$,
        \item $J^2=-id$.
    \end{itemize}
    An almost complex complex structure $J$ is said to be \textbf{compatible} with $\omega$ if 
    \begin{itemize}
        \item $\omega(\cdot,J\cdot)>0$,
        \item $\omega(J\cdot,J\cdot)=\omega(\cdot,\cdot)$.
    \end{itemize}
\end{defn}

\begin{defn}
    Let $(M,\omega)$ be a symplectic manifold with a compatible almost complex structure $J$. A $J$-holomorphic disc is a smooth map $u(s,t):\mathbb{D}\rightarrow M$ such that
    \[
    \frac{du}{ds}+J(u)\frac{du}{dt}=0,
    \]
    where $\mathbb{D}$ is the unit disc in $\C$.
\end{defn}

\begin{defn}
    Let $(M,\omega)$ be a symplectic manifold equipped with a compatible almost complex structure $J$, and let $g_i:L_i\looparrowright M$, $i=1,2$, be two Lagrangian immersions. A $J$-holomorphic disc $u(s,t):\D\rightarrow M$ is said to \textbf{have boundary in $L_1$ and $L_2$}, if $u|_{\partial\D-(\pm1,0)}$ lifts to a smooth map
    \[
    \Tilde{u}:\partial\D-(\pm1,0)\rightarrow L_1\cup L_2
    \]
    such that
    \begin{itemize}
        \item $\Tilde{u}(e^{i\theta})\in L_1$, $\theta\in(0,\pi)$,
        \item $\Tilde{u}(e^{i\theta})\in L_2$, $\theta\in (\pi,2\pi)$,
        \item $\lim_{\theta\rightarrow\pi^-}g_1(\Tilde{u}(e^{i\theta}))=\lim_{\theta\rightarrow\pi^+}g_2(\Tilde{u}(e^{i\theta}))=u(-1,0)$,
        \item $\lim_{\theta\rightarrow0^+}g_1(\Tilde{u}(e^{i\theta}))=\lim_{\theta\rightarrow0^-}g_2(\Tilde{u}(e^{i\theta}))=u(1,0)$.
    \end{itemize}
\end{defn}

\begin{defn}
    Let $g_1:L_1\rightarrow M$ and $g_2:L_2\rightarrow M$ be two Lagrangian immersions into a symplectic manifold $(M,\omega)$. The {\bf generalized intersections of $L_1$ and $L_2$} are points in the fiber product
\[
L_1\times_M L_2=\{(x_1,x_2)\in L_1\times L_2|\ l_1(x_1)=l_2(x_2)\}.
\]
\end{defn}

Let $(M,\omega)$ be a symplectic manifold with a compatible almost complex structure $J$. Suppose that $L_1$ and $L_2$ are two Lagrangian immersions in $M$ that intersect transversely. Let $u:\D\rightarrow M$ be an immersed holomorphic disc with its boundary in $L_1$ and $L_2$ connecting $x_\pm\in L_1\times_M L_2$. Denote by $W^{1,p}_{x_\pm}(\D,M)$ the Sobolev $W^{1,p}$ space of all such immersed discs connecting $x_\pm$, with $p\in(2,\infty)$. For $u\in W^{1,p}_{x_\pm}(\D,M)$, let $L^p(\D,\Omega^{0,1}\otimes u^*TM)$ denote the Sobolev $L^p$ completion of the sections of the bundle $\Omega^{0,1}\otimes u^*TM\rightarrow \D$, where $\Omega^{0,1}$ is the bundle of $(0,1)$-forms over $\D$. Therefore, the Cauchy-Riemann operator induces the map
\begin{equation}\label{equ:3.1}
    \begin{aligned}
    \mathcal{F}: W^{1,p}_{x_\pm}(\D,M)\times \mathcal{J}&\rightarrow L^p(\D,\Omega^{0,1}\otimes u^*TM)\\
    (u,J)&\mapsto \bar\partial_Ju,
\end{aligned}
\end{equation}
where $\mathcal{J}$ is the space of all $J$-holomorphic structures that are compatible with $\omega$. \\

The following theorem is proved with $\D$ replaced by $\R\times[0,1]$ in \cite{zhang2024construction} (Section 3). Since $\D$ and $\R\times[0,1]$ can be identified using Riemann mapping theorem, the theorem below holds.

\begin{them}
    Suppose $(M,\omega)$ is a closed symplectic manifold. Let $L_1$ and $L_2$ be two Lagrangian immersions in $M$. Denote $\mathcal{J}$ as the set of compatible almost complex structures on $M$. Then the differential $D\bar\partial_J$ of the Cauchy-Riemann operator 
    \begin{equation*}
        \bar\partial_J:W^{1,p}_{x_\pm}(\D,M)\rightarrow L^p(\D,\Omega^{0,1}\otimes u^*TM)
    \end{equation*}
    is surjective for an open and dense subset of $\mathcal{J}$, where $x_\pm\in L_1\times_M L_2$.
\end{them}

\begin{defn}
    Denote $\mathcal{J}_{reg}\subset\mathcal{J}$ as the subset where the differential of the Cauchy-Riemann operator $D\bar\partial_J$ is surjective. An element $J\in \mathcal{J}_{reg}$ is called a \textbf{regular almost complex structure}.
\end{defn}

Since $D\bar\partial_J$ is Fredholm with trivial cokernel for $J\subset\mathcal{J}_{reg}$, if moreover, the index of $D\bar\partial_J$ is 0, then space of the solutions of the Cauchy-Riemann operator is a smooth manifold of dimension 0. According to Gromov compactness, this space consists of finitely many points. Thus, we have the following corollary.

\begin{cor}
    Suppose $(M,\omega)$ is a closed symplectic manifold. Let $L_1$ and $L_2$ be two Lagrangian immersions in $M$. If $J\in\mathcal{J}_{reg}$ and the index of the differential of the Cauchy-Riemann operator is 0, then the solution space of the Cauchy-Riemann operator consists of finitely many points.
\end{cor}

Given $g:L\looparrowright M$ as a Lagrangian immersion and $g_1^t:L_t\looparrowright M$ as a smooth family of Lagrangian immersions for $t\in[0,1]$ in a closed symplectic manifold $(M.\omega)$, assume $L_t$ intersects $L$ transversely for all $t$. Let $x_\pm^t$ represent two paths of intersections between $L_t$ and $L$ parametrized by $t$. Define
\[
W^{1,p}_{x_\pm^{[0,1]}}(\D,M):=\bigcup_t W^{1,p}_{x_\pm^t}(\D,M),
\]
and
\[
L^p_{[0,1]}:= \bigcup_tL^p(\D,\Omega^{0,1}\otimes (u^t)^*TM),\ \ \{u^t\}_t\in W^{1,p}_{x_\pm^{[0,1]}}(\D,M).
\]
With these notations, the proof of the theorem below follows the argument from McDuff-Salamon \cite{mcduff2012j} (Theorem 3.1.8) or Floer \cite{floer1988morse} (Theorem 5a).

\begin{them}\label{thm:regfam}
    Let $(M,\omega)$ be a closed symplectic manifold. Suppose that $g:L\looparrowright M$ is a Lagrangian immersion and $g_1^t:L_t\looparrowright M$ is a smooth family of Lagrangian immersions for $t\in[0,1]$. Assume $L_t$ intersects $L$ transversely for all $t$. Given $J_0\in\mathcal{J}_{reg}$, then for generic based loops $\{J_t\}_{t\in[0,1]}\in\mathcal{J}$ starting at $J_0$ in $\mathcal{J}$, we have $\{J_t\}_{t\in[0,1]}\subset\mathcal{J}_{reg}$.
\end{them}

\begin{proof}
    Denote the space of loops of $\mathcal{J}$ based at $J_0$ as $\mathcal{LJ}$. Consider the operator
    \begin{align*}
        \bar\partial_{[0,1]}:W^{1,p}_{x_\pm^{[0,1]}}(\D,M)\times\mathcal{LJ}&\rightarrow L^p_{[0,1]}\\
        (\{u^t\}_t,\{J_t\}_t)&\mapsto \bar\partial_{J_t}u^t.
    \end{align*}
    This is a Fredholm operator with 0 being a regular value. The projection
    \[
    \pi_{\mathcal{LJ}}:W^{1,p}_{x_\pm^{[0,1]}}(\D,M)\times\mathcal{LJ}\supset\bar\partial_{[0,1]}^{-1}(0)\rightarrow\mathcal{LJ}
    \]
    is also Fredholm. According to the Sard-Smale theorem for Bananch spaces (c.f \cite{mcduff2012j} Theorem A.5.1), the regular value of $\pi_{\mathcal{LJ}}$ is generic. Let $\{J_t\}_{t\in[0,1]}$ be a generic value of $\pi_{\mathcal{LJ}}$, then $J_t\in\mathcal{J}_{reg}$, $\forall t\in[0,1]$. This completes the proof of the theorem.
\end{proof}

\section{Pearly tree immersed Lagarangian Floer theory}
This section gives the definition of pearly tree immersed Lagrangian Floer chain group and its boundary map. Some descriptions of the boundary maps are presented at the end of this section.\\

\begin{defn}
    Let $(M^n,\omega)$ be a symplectic manifold and $g:L\looparrowright M$ be a Lagrangian immersion. A \textbf{morse function $f$ of $L\looparrowright M$} is a morse function defined on $L$ whose critical points are different from the self-intersections of $L$. The set of critical points of $f$ is denoted as $\Crit(f)$.
\end{defn}

\begin{defn}
    Let $(M,\omega)$ be a symplectic manifold and $g:L\looparrowright M$ be a Lagrangian immersion. The set of \textbf{ordered pair of self-intersections of $L$} is defined as
    \[
    R:=\{(p,q)\in L\times L|\ g(p)=g(q)\}.
    \]
\end{defn}

\begin{rmk}
    Notice that $(p,q)\ne(q,p)$.
\end{rmk}

\begin{defn}
    Let $(M,\omega)$ be a symplectic manifold and $L\looparrowright M$ be a Lagrangian immersion. Given a morse function $f$ of $L\looparrowright M$, the \textbf{pearly tree Lagrangian Floer chain group} $\CF^P(L)$ is generated by the points in $\Crit(f)$ and the ordered pair of self-intersections of $L$. %Denote the subgroup of $\CF^P(L)$ generated by $\Crit(f)$ as $\CF^P_C(L)$ and the subgroup generated by the ordered pair of self-intersections of $L$ as $\CF^P_{SI}(L)$.
\end{defn}

%\begin{rmk}
%    It is clear that $\CF^P(L)=\CF^P_C(L)\oplus\CF^P_{SI}(L)$.
%\end{rmk}

The definition of the boundary map of a pearly tree Lagrangian Floer chain group follows from Alston-Bao \cite{alston2021immersed}.

\begin{defn}
    Let $(M,\omega)$ be a symplectic manifold equipped with a compatible almost complex structure $J$ and $g:L\looparrowright M$ be a Lagrangian immersion. A $J$-holomorphic disc $u(s,t):\D\rightarrow M$ is said to \textbf{have boundary in $L$}, if $u|_{\partial\D-(\pm1,0)}$ lifts to a smooth map
    \[
    \Tilde{u}:\partial\D-(\pm1,0)\rightarrow L
    \]
    such that 
    \[
    \lim_{\theta\rightarrow\pi^-}g(\Tilde{u}(e^{i\theta}))=\lim_{\theta\rightarrow\pi^+}g(\Tilde{u}(e^{i\theta}))=u(-1,0), \lim_{\theta\rightarrow0^+}g(\Tilde{u}(e^{i\theta}))=\lim_{\theta\rightarrow0^-}g(\Tilde{u}(e^{i\theta}))=u(1,0).
    \]
    The point $u(-1,0)$ $(resp.u(1,0))$ is called \textbf{branch switching}, if $\lim_{\theta\rightarrow\pi^-}\Tilde{u}(e^{i\theta})\ne$ $\lim_{\theta\rightarrow\pi^+}\Tilde{u}(e^{i\theta})$ $(resp.\lim_{\theta\rightarrow0^-}\Tilde{u}(e^{i\theta})\ne\lim_{\theta\rightarrow0^+}\Tilde{u}(e^{i\theta}))$. Otherwise, $u(\pm1,0)$ is called \textbf{branch non-swtiching}.
\end{defn}

\begin{defn}\label{defn:pearlydisc}
    Let $(M,\omega)$ be a symplectic manifold equipped with a compatible almost complex structure $J$ and $g:L\looparrowright M$ be a Lagrangian immersion. Assume that $f$ is a morse function of $L$ such that $\Crit(f)$ does not contain self-intersections of $L$. Given generators $x,y$ of $\CF^P(L)$, the \textbf{pearly tree discs from $x$ to $y$} are categorized into the following 4 types:
    \begin{itemize}
        \item[Type1:] If $x,y\in\Crit(f)$, then the pearly tree disc is a morse trajectory from $x$ to $y$.
        \item[Type2:] If $x\in\Crit(f)$ and $y=(p,q)$ is an ordered pair of self-intersections, then the pearly tree disc is a morse trajectory $l$ followed by a holomorphic disc $u$, i.e.
        \[
        l:(-\infty,a]\rightarrow L,\ u:\D\rightarrow M,\ \ a\in\R
        \]
        such that $l$ starts at $x$, $g(l(a))=u(-1,0)$, $u(-1,0)$ is branch non-switching, and $u(1,0)$ is branch switching. Moreover, the lift $\Tilde{u}$ of $u|_{\partial\D-(\pm1,0)}$ satisfies $\lim_{\theta\rightarrow0^+}\Tilde{u}(e^{i\theta})=p$ and $\lim_{\theta\rightarrow0^-}\Tilde{u}(e^{i\theta})=q$.
        \item[Type3:] If $x=(p,q)$ is an ordered pair of self-intersections and $y\in\Crit(f)$, then the pearly tree disc is a holomorphic disc $u$ followed by a morse trajectory $l$, i.e.
        \[
        u:\D\rightarrow M,\ l:[a,+\infty)\rightarrow L,\ \ a\in\R
        \]
        such that $l$ ends at $y$, $g(l(a))=u(1,0)$, $u(-1,0)$ is branch switching, and $u(1,0)$ is branch non-switching. Moreover, the lift $\Tilde{u}$ of $u|_{\partial\D-(\pm1,0)}$ satisfies $\lim_{\theta\rightarrow\pi^-}\Tilde{u}(e^{i\theta})=p$ and $\lim_{\theta\rightarrow\pi^+}\Tilde{u}(e^{i\theta})=q$.
        \item[Type4:] If $x=(p,q)$ and $y=(p',q')$ are both ordered pair of self-intersections, then the pearly tree disc is a holomorphic disc $u$ such that $u(\pm1,0)$ are both branch switching. Moreover, the lift $\Tilde{u}$ of $u|_{\partial\D-(\pm1,0)}$ satisfies $\lim_{\theta\rightarrow\pi^-}\Tilde{u}(e^{i\theta})=p$, $\lim_{\theta\rightarrow\pi^+}\Tilde{u}(e^{i\theta})=q$, $\lim_{\theta\rightarrow0^-}\Tilde{u}(e^{i\theta})=p'$, and $\lim_{\theta\rightarrow0^+}\Tilde{u}(e^{i\theta})=q'$.
    \end{itemize}
    The moduli space of the above 4 types of pearly tree discs connecting $x$ and $y$ modulo the $\R$ translation is denoted as $\mathcal{M}^P(x,y)$.
\end{defn}

\begin{figure}[H]
\centering 
\includegraphics[width=0.8\textwidth]{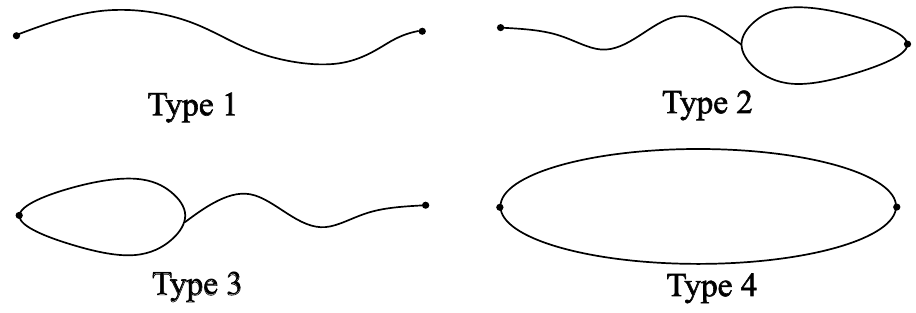}
\caption{The four types of pearly tree discs. The left dots are inputs and the right dots are outputs.}
\end{figure}

\begin{defn}
    Let $(M,\omega)$ be a symplectic manifold equipped with an almost complex structure $J\in\mathcal{J}_{reg}$ and $g:L\looparrowright M$ be a Lagrangian immersion. Given a morse function of $L$, the boundary map of the corresponding pearly tree Lagrangian Floer chain group $(\CF^P(L),\partial^P)$ is defined as
    \begin{align*}
        \partial^P:\CF^P(L)&\rightarrow \CF^P(L)\\
        \partial^P(x)&=\sum_{y\in\CF^P(L),\dim\mathcal{M}^P(x,y)=0}\#\mathcal{M}^P(x,y)y.
    \end{align*}
\end{defn}

To understand the Type 2 and Type 3 pearly tree discs in Definition \ref{defn:pearlydisc}, we have the following theorem.

\begin{them}
    Let $(M,\omega)$ be a symplectic manifold equipped with an almost complex structure $J\in\mathcal{J}_{reg}$ and $g:L\looparrowright M$ be a Lagrangian immersion. Given a morse function of $L$, denote the corresponding pearly tree Lagrangian Floer chain group as $\CF^P(L)$. Suppose $x,y$ are generators of $\CF^P(L)$.
    \begin{itemize}
        \item If there is a morse trajectory starting at $x$ and the moduli space $\mathcal{M}^P(x,y)\ne\emptyset$ is of Type 2 in Definition \ref{defn:pearlydisc}, then $\dim\mathcal{M}^P(x,y)\ge1$. 
        \item If there is a morse trajectory ending at $y$ and the moduli space $\mathcal{M}^P(x,y)\ne\emptyset$ is of Type 3 in Definition \ref{defn:pearlydisc}, then $\dim\mathcal{M}^P(x,y)\ge1$. 
    \end{itemize}
\end{them}

\begin{proof}
    Because the arguments for the both claims in the theorem are similar, the proof of the first one is presented here. The idea of proof is derived from Floer \cite{floer1988morse}. More specifically, suppose that there is a morse trajectory starting at $x$. Then we can move the point connecting the morse trajectory and the holomorphic disc along the morsed trajectory. Denote the trajectory of this connecting point as $w(t)$. This defines a family of holomorphic discs connecting $w(t)$ and $y$. Therefore, the dimension of the moduli space is at least 1.\\
    Since by assumption $\mathcal{M}^P(x,y)$ is non-empty, there is a pearly trajectory of Type 2 in Definition \ref{defn:pearlydisc}. Let $l:(-\infty,a]\rightarrow M$ be the morse trajectory part and $u:\D\rightarrow M$ be the holomorphic disc part. Denote the connecting point $l(a)$ of $l$ and $u$ as $w(0)$. Then $u$ is a holomorphic disc connecting $w(0)$ and $y$. Moving $w(0)$ along the morse trajectory $l$, we obtain a path $w(t)$ starting at $w(0)$. Define $\Tilde{\mathcal{M}}(w(t),y)$ as the space of holomorphic discs connecting $w(t)$ and $y$ with boundary in $L$.
    
    Before analyzing the dimension of the moduli space, some transversality results are needed. Let $W^{1,p}_{w(t),y}(\D,M)$ be the Sobolev $W^{1,p}$ space of all immersed holomorphic discs with boundary in $L$ connecting $w(t)$ and $y$, for $p\in(2,\infty)$. For $u^t\in W^{1,p}_{w(t),y}(\D,M)$, let $L^p(\D,\Omega^{0,1}\otimes (u^t)^*TM)$ denote the Sobolev $L^p$ completion of the sections of the bundle $\Omega^{0,1}\otimes (u^t)^*TM\rightarrow \D$, where $\Omega^{0,1}$ is the bundle of $(0,1)$-forms over $\D$. Then $\Tilde{\mathcal{M}}(w(t),y)$ is the solution space of the Cauchy-Riemann equation from $W^{1,p}_{w(t),y}(\D,M)$ to $L^p(\D,\Omega^{0,1}\otimes (u^t)^*TM)$. Define
    \[
    W^{1,p}_{w([0,b)),y}(\D,M):=\bigcup_{t\in[0,b)}W^{1,p}_{w(t),y}(\D,M),\quad L^p_{[0,b)}:=\bigcup_{t\in[0,b)}L^p(\D,\Omega^{0,1}\otimes (u^t)^*TM).
    \]
    The above spaces are Banach manifolds, and the Cauchy-Riemann operator 
    \begin{align*}
        \mathcal{F}:W^{1,p}_{w([0,b)),y}(\D,M)&\rightarrow L^p_{[0,b)}\\
        u^t&\mapsto\bar\partial_Ju^t
    \end{align*}
    is regular at $t=0$, since $J\in\mathcal{J}_{reg}$ by assumption. More precisely, the linearization of $\mathcal{F}$ is surjective at $t=0$. Because this surjectivity is an open condition, there is a small enough number $\delta>0$ such that the linearization of $\mathcal{F}$ is surjective for $0\le t<\delta$. That is, $\mathcal{F}$ is regular for $t\in[0,\delta)$. Therefore, the space $\Tilde{\mathcal{M}}(w([0,\delta)),y):=\mathcal{F}^{-1}(0)=\bigcup_{t\in[0,\delta)}\Tilde{\mathcal{M}}(w(t),y)$ is a smooth manifold when $b=\delta$.

    The space $\Tilde{\mathcal{M}}(w(0),y)$ is non-empty by assumption. Since $J$ is regular and by the argument from the previous paragraph, we have:
    \[
    \dim\Tilde{\mathcal{M}}(w(t),y)\ge1,\ \ t\in[0,\delta).
    \]
    Thus, $\dim\Tilde{\mathcal{M}}(w([0,\delta)),y)\ge2$. By modding out the $\R$-translation from $\Tilde{\mathcal{M}}(w([0,\delta)),y)$, the resulting moduli space $\mathcal{M}(w([0,\delta)),y)$ has dimension at least 1. Since $\mathcal{M}(w([0,\delta)),y)$ is a part of the molduli space of $\mathcal{M}^P(x,y)$, we conclude that $\dim\mathcal{M}^P(x,y)\ge1$.
\end{proof}

\begin{cor}\label{cor:nomorsetra}
    Let $(M^n,\omega)$ be a symplectic manifold equipped with an almost complex structure $J\in\mathcal{J}_{reg}$ and $g:L\looparrowright M$ be a Lagrangian immersion. Given a morse function $f$ of $L$, denote the corresponding pearly tree Lagrangian Floer chain group as $\CF^P(L)$. Suppose $x,y$ are generators of $\CF^P(L)$.
    \begin{itemize}
        \item If the morse index of $x\in\Crit(f)$ is strictly greater than $0$ and $y$ is an ordered pair of self-intersections, then the pearly tree discs from $x$ to $y$ does not contribute to the boundary map of $\CF^P(L)$. In the case where the index of $x\in\Crit(f)$ is $0$, then the morse trajectory part of the pearly tree disc from $x$ to $y$ is constant.
        \item If the morse index of $y\in\Crit(f)$ is strictly smaller than $n$ and $x$ is an ordered pair of self-intersections, then the pearly tree discs from $x$ to $y$ does not contribute to the boundary map of $\CF^P(L)$. In the case where the index of $x\in\Crit(f)$ is $n$, then the morse trajectory part of the pearly tree disc from $x$ to $y$ is constant.
    \end{itemize}
\end{cor}

\section{Hamiltonian immersed Lagrangian Floer theory}
This section gives the definition of Hamiltonian immsersed Lagrangian Floer chain group and its boundary map. The following theorme is the starting point.

\begin{them}[Weinstein Tubular Neighbourhood Theorem]\label{thm:tubu}{\rm \cite{eliashberg2002introduction}}
Let $(M,\omega)$ be a symplectic manifold. Assume that $f:L\rightarrow M$ is a Lagrangian immersion. Then there is a local symplectomorphism $G$ from a tubular neighbourhood $T^*_\varepsilon L$ of the zero section of the cotangent bundle $T^*L$ to a neighbourhood $U\supset f(L)$ as in the following commutative diagram:
\begin{equation*}
    \xymatrix{
    T_\varepsilon^*L \ar[dr]^G & \\
    L \ar[r]^{f\ \ \ } \ar[u] &U\subset M.}
\end{equation*}
\end{them}

\begin{defn}
    Let $(M,\omega)$ be a symplectic manifold and $g:L\looparrowright M$ be a Lagrangian immersion. Let $U$ be an open neighbourhood of $L$ that can be identified with $T_\varepsilon^*L$ as in Theorem \ref{thm:tubu}. A \textbf{local Hamiltonian function} is a smooth function $H:T_\varepsilon^*L\rightarrow\R$. The function $H$ is called \textbf {non-degenerate} if $H|_L$ is a morse function and the self-intersections of $L$ are not the critical points of $H|_L$. Denote the set of all non-degenerate local Hamiltonian functions that are \textbf{constant in the fiber directions} as $\mathcal{H}^{cf}$.
\end{defn}

\begin{defn}
   Let $(M,\omega)$ be a symplectic manifold and $g:L\looparrowright M$ be a Lagrangian immersion. Let $U$ be an open neighbourhood of $L$ that can be identified with $T_\varepsilon^*L$ as in Theorem \ref{thm:tubu}. Given a local Hamiltonian function $H:T_\varepsilon^*L\rightarrow\R$, the \textbf{local Hamiltonian vector field} $X$ of $H$ is a vector field over $T_\varepsilon^*L$ such that
   \[
   \iota_X(G^*\omega)=dH.
   \]
   The \textbf{local Hamiltonian flow} $\{\phi_t\}_{t\ge0}$ of $H$ is a smooth family of diffeomorphisms of $T_\varepsilon^*L$ such that
   \[
   \Dot{\phi_t}=X.
   \]
   Assume that $f:N\rightarrow M$ is an immersion, then $\phi_\epsilon\circ f:N\rightarrow M$ is the immersion after performing the Hamiltonian flow. To distinguish these two domains, we use $N_{\phi_\epsilon}$ to denote the domain of $\phi_\epsilon\circ f$.
\end{defn}

\begin{defn}
    Let $(M,\omega)$ be a symplectic manifold and $g:L\looparrowright M$ be a Lagrangian immersion. Given a Hamiltonian flow $\phi_t$ on $M$ such that $\phi_\epsilon(g(L))$ and $g(L)$ intersects transversely, the \textbf{Hamiltonian Lagrangian Floer chain group $\CF^H(L,L_{\phi_\epsilon})$} is generated by the generalized intersections of $L_{\phi_\epsilon}$ and $L$.
\end{defn}

\begin{defn}
    Let $(M,\omega)$ be a symplectic manifold equipped with a compatible almost complex structure $J$ and $g:L\looparrowright M$ be a Lagrangian immersion. Suppose $\{\phi_t\}$ is a family of Hamiltonian flow. Let $u(s,t):\D\rightarrow M$ be a $J$-holomorphic disc having boundary in $L$ and $L_{\phi_\epsilon}$. Assume $u(-1,0)=x$ and $u(1,0)=y$, then $u$ is said to \textbf{connecting $x$ and $y$}. The moduli space of all $J$-holomorphic disc having boundary in $L$ and $L_{\phi_\epsilon}$ connecting $x$ and $y$ is denoted as $\Tilde{\mathcal{M}}^H(x,y)$. Define $\mathcal{M}^H(x,y)$ as the quotient of $\Tilde{\mathcal{M}}^H(x,y)$ by the $\R$-translation.
\end{defn}

\begin{defn}
    Let $(M,\omega)$ be a symplectic manifold equipped with a compatible almost complex structure $J$ and $g:L\looparrowright M$ be a Lagrangian immersion. Assume $\CF^H(L,L_{\phi_\epsilon})$ is the Hamiltonian Lagrangian Floer chain group defined by the family of Hamiltonian flow $\{\phi_t\}$. The boundary map of $\CF^H(L,L_{\phi_\epsilon})$ is defined as
    \begin{align*}
        \partial^H:\CF^H(L,L_{\phi_\epsilon})&\rightarrow\CF^H(L,L_{\phi_\epsilon})\\
        \partial^H(x)&=\sum_{y\in\CF^H(L),\dim\mathcal{M}^H(x,y)=0}\#\mathcal{M}^H(x,y)y.
    \end{align*}
\end{defn}

\section{Proof of the main theorem}
This section aims to prove the main theorem. A natural step is to construct a map from the pearly tree Lagrangian Floer theory to the Hamiltonian Lagrangian Floer theory. Let $f$ be the morse function of $L$ used to define the pearly tree Lagrangian Floer chain group $(\CF^P(L),\partial^P)$. The local Hamiltonian function required to define the Hamiltonian Lagrangian Floer theory is obtained by extending $f$ constantly in the normal direction within the Weinstein tubular neighbourhood of $L$.

The reverse direction is more challenging, as not all local Hamiltonian functions are constant extensions of a morse function defined on $L$ in the Weinstein tubular neighbourhood of $L$. Suppose the given local Hamiltonian Lagrangian Floer chain group is defined by a non-degenerate local Hamiltonian function $H$. Let $(\CF^P(L),\partial^P)$ be the pearly tree Lagrangian Floer chain group defined by $H|_L$, since $H|_L$ is a morse function of $L$. The idea is to construct an intermediate Hamiltonian Lagrangian Floer chain group using a new local Hamiltonian function $H_1\in\mathcal{H}^{cf}$, where $H_1$ is the constant extension of $H|_L$ in the fiber direction within the Weinstein tubular neighbourhood of $L$. Denote $\phi_t^1$ as the local Hamiltonian flow of $H_1$. We show that $(\CF^{H_1}(L,L_{\phi_\epsilon^1}),\partial^{H_1})$ can be identified with both $(\CF^P(L),\partial^P)$ and $(\CF^{H}(L,L_{\phi_\epsilon}),\partial^{H})$ for $\epsilon>0$ sufficently small, where $\phi_t$ is the local Hamiltonian flow of $H$.

To identify $(\CF^{H_1}(L,L_{\phi_\epsilon^1}),\partial^{H_1})$ and $(\CF^{H}(L,L_{\phi_\epsilon}),\partial^{H})$, the key result is Lemma \ref{lem:famlag}. Lemma \ref{lem:famlag} describes the behavior of holomorphic curves with boundary in $L$ and $L_t$ as $t$ varies, where $\{L_t\}_{t\in[0,1]}$ is a smooth family of Lagrangian immersions in $(M,\omega)$ starting at $L_{\phi_\epsilon}$ and ending at $\phi_{L_\epsilon^1}$. The proof of Lemma \ref{lem:famlag} follows the arguments in Floer \cite{floer1988morse}.

To prove the identification of $(\CF^P(L),\partial^P)$ and $(\CF^{H_1}(L,L_{\phi_\epsilon^1}),\partial^{H_1})$, the bijection of the generators of the pearly tree Lagrangian Floer chain group can be naturally constructed. Subsequently, we show that the boundary maps of these two chain groups are bijective. The proof is divided into three cases: 
\begin{itemize}
    \item The boundary map in the pearly tree Lagrangian Floer chain group counts pearly tree discs of Type 1 in Definition \ref{defn:pearlydisc}.
    \item The boundary map in the pearly tree Lagrangian Floer chain group counts pearly tree discs of Type 2 and 3 in Definition \ref{defn:pearlydisc}.
    \item The boundary map in the pearly tree Lagrangian Floer chain group counts pearly tree discs of Type 4 in Definition \ref{defn:pearlydisc}.
\end{itemize}
For the first case, he proof follows the argument that the Lagrangian Floer homology of a Lagrangian submanifold $L$ is isomorphic to the morse homology $H(L)$. 

For of the second and third cases, the proof uses the following strategy. Starting from the pearly tree holomorphic disc (there are no morse trajectory components by Corollary \ref{cor:nomorsetra}), we move the Lagrangian immersion using a Hamiltonian flow. This generates a family of holomorphic discs. We demonstrate that this family forms a moduli space of dimension 1, and this moduli space connects the pearly tree discs to the holomorphic discs corresponding to the boundary maps of the Hamiltonian Lagrangian Floer chain group.\\

For the reader's convenience, the main theorem is restated here.
\begin{them}
    Let $(M^n,\omega)$ be a closed symplectic manifold. Suppose $g:L\looparrowright M$ be a Lagrangian immersion. Assume $f$ is a morse function defined on $L$ whose critical points are different from the self-intersections of $L$. Given a regular almost complex structure on $M$, then there is a local Hamiltonian flow $\phi_t$ such that the pearly tree immersed Lagrangian Floer chain group $(\CF^P(L),\partial^P)$ defined by $f$ is identified with the Hamiltonian immersed Lagrangian Floer chain group $(\CF^H(L,L_{\phi_\epsilon}),\partial^H)$, where $L_{\phi_\epsilon}$ stands for the Lagrangian immersion $\phi_\epsilon\circ g$ for some small $\epsilon>0$.
    
    Conversely, equipped with a regular almost complex structure on $M$, given a local Hamiltonian flow $\phi_t$ such that the corresponding Hamiltonian function restricting to $L$ is a morse function $f$, then there is a canonical identification between $(\CF^H(L,L_{\phi_\epsilon}),\partial^H)$ and the pearly tree immersed Lagrangian Floer chain group $(\CF^P(L),\partial^P)$ defined by $f$, for small enough $\epsilon>0$.
\end{them}

\begin{proof}
    It is a direct consequence of Theorem \ref{thm:5.1} and Theorem \ref{thm:5.2}.
\end{proof}

\begin{them}\label{thm:5.1}
    Let $(M,\omega)$ be a closed symplectic manifold. Suppose $g:L\looparrowright M$ is a Lagrangian immersion, and assume $f$ is a morse function defined on $L$ whose critical points are distinct from the self-intersections of $L$. Given a regular almost complex structure on $M$, there is a local Hamiltonian flow $\phi^1_t$ associated to $H_1\in\mathcal{H}^{cf}$, such that the pearly tree immersed Lagrangian Floer chain group $(\CF^P(L),\partial^P)$ defined by $f$ is isomorphic to $(\CF^{H_1}(L,L_{\phi^1_\epsilon}),\partial^{H_1})$ for some sufficiently small $\epsilon>0$.
    
    Conversely, equipped with a regular almost complex structure on $M$, given a local Hamiltonian flow $\phi_t$ associated with $H_1\in\mathcal{H}^{cf}$, there is a canonical isomorphism between $(\CF^{H_1}(L,L_{\phi^1_\epsilon}),\partial^{H_1})$ and the pearly tree immersed Lagrangian Floer chain group $(\CF^P(L),\partial^P)$ defined by $f:=H_1|_{L}$ for small enough $\epsilon>0$.
\end{them}

\begin{proof}
    \textbf{Step 1: identification of the generators.}\par
    Recall that $H_1|_L=f$ and $H_1$ is the constant extension of $f$ in the fiber direction of $T_\varepsilon^*L$. The goal is to show that the critical points $f$ are the same as the fixed points of the local Hamiltonian flow $\phi^1_t$ of $H_1$ when $t$ is small. Let $(x_1,\ldots,x_n)$ be local coordinates of $L$ and $(x_1,\ldots,x_n,y_1,\ldots,y_n)$ be the corresponding local coordinates of $T_\varepsilon^*L$. In these local coordinates, the symplectic form is expressed as $\omega=\sum_{i=1}^ndx_i\wedge dy_i$. So the local Hamiltonian vector field $X$ of $H_1$ is
    \[
    X=\sum_{i=1}^n\frac{\partial H_1}{\partial x_i}\frac{\partial}{\partial y_i}.
    \]
    Given $x\in L$, notice that $X|_x=0$ if and only if 
    $$
    \sum_{i=1}^n\frac{\partial H_1}{\partial x_i}\frac{\partial}{\partial y_i}|_x=\sum_{i=1}^n\frac{\partial f}{\partial x_i}\frac{\partial}{\partial y_i}|_x=0.
    $$
    Therefore the points where $X|_x=0$ are exactly the critical points of $f$. Moreover, the fixed points of $\phi^1_t$ are $x\in L$ such that $X|_x=0$. This establishes the bijection between the critical points $f$ and the fixed points of $H_1$.\par
    
    The remaining the generators of $(\CF^P(L),\partial^P)$ and $(\CF^{H_1}(L,L_{\phi^1_\epsilon}),\partial^{H_1})$ are both from the self-intersections of $L$. Each self-intersection of $L$ contributes two generators to both $(\CF^P(L),\partial^P)$ and $(\CF^{H_1}(L,L_{\phi^1_\epsilon}),\partial^{H_1})$. This completes step 1.\\

    \textbf{Step 2: identification of the boundary maps.}\par
    Case 1. The boundary map of $(\CF^P(L),\partial^P)$ counting the pearly tree discs of Type 1 in Definition \ref{defn:pearlydisc}. Suppose $u(s)$ is a morse trajectory connecting $x,y\in\Crit(f)$, then
    \[
    \frac{du(s)}{ds}=-\mathrm{grad}(f)(u(s)).
    \]
    Let $\phi^1_t$ be the local Hamiltonian flow of $H_1$, the following calculation shows that $\widetilde{u}(s,t):=\phi^1_t(u(s))$ is a holomorphic map from $\R\times[0,\epsilon]$ to $T_\varepsilon^*L$.
    \begin{equation*}
        \begin{aligned}
            \frac{\partial}{\partial s}\widetilde{u}(s,t)+J\frac{\partial}{\partial t}\widetilde{u}(s,t)&=\frac{\partial}{\partial s}\phi^1_t(u(s))+J\frac{\partial}{\partial t}\phi^1_t(u(s))\\
            &=d\phi^1_t(u(s))(\frac{d}{ds}u(s))+Jd\phi^1_t(u(s))(X(u(s)))\\
            &=-d\phi^1_t(u(s))(\mathrm{grad}(f)(s))+Jd\phi^1_t(u(s))(J\mathrm{grad}(f)(u(s)))\\
            &=-d\phi^1_t(u(s))(\mathrm{grad}(f)(s))-J^2d\phi^1_t(u(s))(\mathrm{grad}(f)(u(s)))\\
            &=0,
        \end{aligned}
    \end{equation*}
    where third equation above holds because $\sum_{i=1}^ndx_i\wedge dy_i(X,\cdot)=dH_1$,
    \[
    \mathrm{grad}(H_1)=-JX,
    \]
    and $H_1|_L=f$. Therefore we get a holomorphic disc by identifying $\R\times[0,\epsilon]$ with $\D$.\\
    Conversely, let $\widetilde{u}:\D\rightarrow T_\varepsilon^*L$ be a holomorphic disc connecting two fixed points of the local Hamiltonian flow of $H_1$. Identifying $\D$ with $\R\times[0,1]$, we assume the domain of $\widetilde{u}$ is $\R\times[0,1]$. \textbf{To proceed, we require the following assumption:}
    \begin{equation}\label{equ:5.1}
        |H_1|+\sum_i|\frac{\partial H_1}{\partial x_i}|+\sum_{i.j}|\frac{\partial^2H_1}{\partial x_i\partial x_j}|\le \varepsilon_1,
    \end{equation}
    for some $\varepsilon_1>0$ small enough. Because $f=H_1|_L$, the function $f$ and its first and second derivatives are controlled by $\varepsilon_1$. Applying the argument after Lemma 5.1 in Floer \cite{floer1989witten}, and assuming Equation \ref{equ:5.1}, we conclude that $\widetilde{u}(s,t)$ can be expressed in the form of $\phi_t^1(u(s))$. Then
    \begin{equation*}
        \begin{aligned}
            0&=\frac{\partial}{\partial s}\widetilde{u}(s,t)|_{t=0}+J\frac{\partial}{\partial t}\widetilde{u}(s,t)|_{t=0}\\
            &=\frac{\partial}{\partial s}\phi^1_t(u(s))|_{t=0}+J\frac{\partial}{\partial t}\phi^1_t(u(s))|_{t=0}\\
            &=d\phi^1_t(u(s))(\frac{d}{ds}u(s))|_{t=0}+Jd\phi^1_t(u(s))(X(u(s)))|_{t=0}\\
            &=\frac{d}{ds}u(s)+JX(u(s))\\
            &=\frac{d}{ds}u(s)-J^2\mathrm{grad}(H_1)(u(s)).
        \end{aligned}
    \end{equation*}
    Therefore
    \[
    \frac{d}{ds}u(s)=-\mathrm{grad}(H_1)(u(s)).
    \]
    To remove the assumption in Equation \ref{equ:5.1}, observe that there is a number $a$ small enough such that
    \begin{equation*}
        |aH_1|+\sum_i|\frac{\partial aH_1}{\partial x_i}|+\sum_{i.j}|\frac{\partial^2aH_1}{\partial x_i\partial x_j}|\le \varepsilon_1.
    \end{equation*}
    Then there is a bijection between the Type 1 pearly tree discs (i.e. morse trajectories) of $(\CF^P(L),\partial^P)$ and the holomorphic discs connecting the fixed points of the local Hamiltonian flow defined by $aH_1$ counting the boundary map of the Hamiltonian immersed Lagrangian Floer chain group induced by $aH_1$. By solving Equation \ref{equ:5.2}, 
    \[
    \phi_t^1|_L(x)=(x,tdH_1).
    \]
    Similarly, the local Hamiltonian flow of $aH_1$ is expressed as
    \[
    \phi_{at}^1|_L(x)=(x,atdH_1).
    \]
    Therefore, there is a smooth family of Lagrangian immersions from $\phi_{at}^1|L$ to $\phi_{t}^1|L$ satisfying the conditions of Lemma \ref{lem:famlag}. Consiquently, the boundary maps of the Hamiltonian immersed Lagrangian Floer chain groups defined by $aH_1$ and $H_1$ are identified. This completes the proof of Case 1.\par
    
    Case 2. The boundary map of $(\CF^P(L),\partial^P)$ counting the pearly tree discs of types other than Type 1 in Definition \ref{defn:pearlydisc}. According to Corollary \ref{cor:nomorsetra}, pearly discs of Types 2, 3, and 4 that contribute to the boundary maps are holomorphic discs. This scenario is equivalent to the case where a holomorphic disc has its boundary on two different Lagrangians.

    To establish a bijection between the boundary maps of $(\CF^P(L),\partial^P)$ and $(\CF^{H_1}(L,L_{\phi^1_\epsilon})$. By moving one piece of the boundary of these holomorphic discs using the local Hamiltonian flow of $H_1$ and applying Lemma \ref{lem:famlag}, we get the bijection of the boundary maps of $(\CF^P(L),\partial^P)$ and $(\CF^{H_1}(L,L_{\phi^1_\epsilon}),\partial^{H_1})$.
\end{proof}

\begin{lem}\label{lem:famlag}
    Let $(M,\omega)$ be a closed symplectic manifold. Given an almost complex structure $J\in\mathcal{J}_{reg}$. Suppose that $g:L\looparrowright M$ is a Lagrangian immersion and $g_1^t:L_t\looparrowright M$ is a smooth family of Lagrangian immersions for $t\in[0,1]$. Assume $L_t$ and $L$ intersect transversely for all $t$. Then the holomorphic discs with boundary in $L_0$ and $L$ are identified with the holomorphic discs in $L_1$ and $L$, where all the holomorphic discs here are assumed to have 0-dimensional moduli spaces.
\end{lem}

\begin{proof}
    Choose a regular path $J_\lambda\in\mathcal{J}_{reg}$, $\lambda\in[0,1]$ such that $J_0=J_1$ (Theorem \ref{thm:regfam}). Because $L_t$ and $L$ intersect transversely for all $t$, the intersections of $L_0$ with $L$ and $L_1$ with $L$ are identified via the path defined by the intersections of $L_t$ with $L$ as $t$ varies. Suppose $x_\pm^0$ is a pair of intersections $L_0$ and $L$ connected by a holomorphic disc with boundary in $L_0$ and $L$, and assume the moduli space of these holomorphic disc here has dimension 0. Let $x^t_\pm$ be the paths of intersections of $L_t$ with $L$, starting at $x_\pm^0$ as $t$ varies. To complete the proof, we need to show that there is an identification between all the holomorphic discs with 0-dimensional moduli space connecting $x_\pm^0$ and all the holomorphic discs with 0-dimensional moduli space connecting $x_\pm^1$. Consider the map
    \begin{align*}
        \bar\partial_{[0,1]}:[0,1]\times W^{1,p}_{x_\pm^t}(\D,M)&\rightarrow L^p(\D,\Omega^{0,1}\otimes u^*TM)\\
        (t,u)&\mapsto\bar\partial_{J_t}u.
    \end{align*}
    By the assumption on $J_t$, 0 is a regular value of the above map. Moreover, for every $t$, $\bar\partial_{J_t}^{-1}(0)$ is of dimension 0. Consequently, the space $\bar\partial^{-1}_{[0,1]}(0)$ is a smooth family of points parametrized by $[0,1]$. This family is non-empty because there is a holomorphic disc connecting $x^0_\pm$ with boundary in $L_0$ and $L$. As a result, there is a holomorphic disc connecting $x^1_\pm$ with boundary in $L_1$ and $L$. This identifies all the holomorphic discs with 0-dimensional moduli space connecting $x_\pm^0$ with all the holomorphic discs with 0-dimensional moduli space connecting $x_\pm^1$. So the Lemma follows.
\end{proof}

\begin{them}\label{thm:5.2}
    Let $(M,\omega)$ be a closed symplectic manifold. Suppose that $g:L\looparrowright M$ is a Lagrangian immersion. Suppose $H$ is a non-degenerate local Hamiltonian function and $H_1\in\mathcal{H}^{cf}$ is the constant extension along the fiber direction of $H_L$ in the Weinstein tubular neighbourhood of $L$. Let $\phi_t$ and $\phi_t^1$ be the local Hamiltonain flow of $H$ and $H_1$, respectively. Given an almost complex structure $J\in\mathcal{J}_{reg}$, then there is a small number $\delta>0$ such that $(\CF^{H}(L,L_{\phi_\epsilon}),\partial^{H})$ and $(\CF^{H_1}(L,L_{\phi_\epsilon^1}),\partial^{H_1})$ are isomorphic for all $\epsilon\in(0,\delta)$.
\end{them}

\begin{proof}
    \textbf{Step 1: identification of the generators.}\par
    The generators of $(\CF^{H}(L,L_{\phi_\epsilon}),\partial^{H})$ and $(\CF^{H_1}(L,L_{\phi_\epsilon^1}),\partial^{H_1})$ that are induced from the self-intersections $L$ can always be identified when $\epsilon$ is small. So we only need to construct a bijection between the remaining generators. Consider a smooth path connecting $\phi_t$ and $\phi_t^1$ defined as
    \begin{align*}
        \mathcal{P}:[0,1]\times L&\rightarrow T_\varepsilon^*L\\
        \mathcal{P}_{s,t}(x)&=\phi_{(1-s)t}\circ\phi_{st}^1(x).
    \end{align*}
    If $\mathcal{P}_{\cdot,t}(L)$ intersects $L$ transversely and $t$ is a small enough number, then $\mathbf{P}_{\cdot,t}^{-1}(L)$ is a 1-dimensional manifold connecting all the intersections of $\phi_t(L)$ with $L$ and the intersections of $\phi_t^1(L)$ and $L$. Therefore, the generators of 
    $(\CF^{H}(L,L_{\phi_\epsilon}),\partial^{H})$ and $(\CF^{H_1}(L,L_{\phi_\epsilon^1}),\partial^{H_1})$ are identified.\par

    To prove the transversality, we need to calculate the tangent map of $\phi_t$ and $\phi_t^1$. Let $(x_1,\ldots,x_n,y_1,\ldots,y_n)$ be the Darboux coordinates around $x_0\in L\subset T_\varepsilon^*L$. Here $(x_1,\ldots,x_n)$ and $(y_1,\ldots,y_n)$ are the coordinates for $L$ and the fiber, respectively. The Hamiltonian flow $\phi_t$ is defined by the following ODE:
    \begin{equation*}
    \left\{
        \begin{aligned}
            \Dot{x}=&\ \ \ \,\frac{\partial H}{\partial y}\\
            \Dot{y}=&-\frac{\partial H}{\partial x}
        \end{aligned}
    \right..
    \end{equation*}
    Let $v$ be tangent vector of $L$ at $x_0$. Define
    \[
    v(t):=d\phi_t(v).
    \]
    By taking the derivative with respect to $t$,
    \begin{equation*}
        \Dot{v}(t)=d\Dot{\phi_t}v=
        \left[
        \begin{aligned}
            &\frac{\partial^2H}{\partial x\partial y}\ \ \ \ \ \,\frac{\partial^2H}{\partial y^2}\\
            -&\frac{\partial^2H}{\partial x^2}\ \ -\frac{\partial^2H}{\partial x\partial y}
        \end{aligned}
        \right]v.
    \end{equation*}
    This is an ODE with initial value ${v}(0)=v$.
    Similarly, $\phi_t^1$ is the solution of 
    \begin{equation}\label{equ:5.2}
    \left\{
        \begin{aligned}
            \Dot{x}=&\ \ \ \,\frac{\partial H_1}{\partial y}(=0)\\
            \Dot{y}=&-\frac{\partial H_1}{\partial x}
        \end{aligned}
    \right..
    \end{equation}
    The derivative of $v(t):=d\phi_t^1(v)$ with respect to $t$ is
    \begin{equation*}
        \Dot{v^1}(t)=d\Dot{\phi_t^1}v=
        \left[
        \begin{aligned}
            &0\ \ \ \ \ \ \ 0\\
            -&\frac{\partial^2H_1}{\partial x^2}\ \ 0
        \end{aligned}
        \right]v.
    \end{equation*}
    Notice that $d\phi_0=Id$, for $v\in T_{x_0}L$,
    \[
    \pi_y\frac{d}{dt}|_{t=0}d(\phi_{(1-s)t}\circ\phi_{st}^1)(v)=\pi_yd((1-s)\Dot{\phi_0}+s(d\phi_0)\Dot{\phi_0^1})(v)=-\frac{\partial^2H}{\partial x^2}|_{x_0}(v),
    \]
    where $\pi_y$ is the projection to the fiber direction $F$.\par

    Next we show that there is a small interval $(0,\delta_1)$ such that
    \begin{equation}\label{equ:fibermap}
        \pi_yd(\phi_{(1-s)t}\circ\phi_{st}^1):T_{x_0}L\rightarrow T_{x_0}F
    \end{equation}
    is isomorphic. Suppose on the contrary, there is  a sequence $t_i\searrow0$ such that the above map is not isomorphic at $t_i$. According to the mean value theorem, there are $\theta_i\in(t_{i},t_{i+1})$ for $i=1,2,\ldots,$ such that 
    \[
    \det(\frac{d}{dt}|_{\theta_i}\pi_yd(\phi_{(1-s)t}\circ\phi_{st}^1))=0.
    \]
    When $i\rightarrow\infty$,
    \[
    \det(-\frac{\partial^2H}{\partial x^2}|_{x_0})=0.
    \]
    Since $H$ is non-degenerate, $\frac{\partial^2H}{\partial x^2}$ is invertible. This gives a contradiction. As a result, for every $x_0\in L$, there is a neighbourhood $U\subset L$ of $x_0$ and a positive number $\delta_1>0$ such that the map defined in Equation \ref{equ:fibermap} is isomorphism for $t\in(0,\delta_1)$. Therefore, $\mathcal{P}_{s,t}|_U$ intersects $L$ transversely for all $t\in(0,\delta_1)$ and $s\in[0,1]$. Combining this with the fact that $L$ is compact, there is a number $\delta>0$ such that $\mathcal{P}_{s,t}(L)$ intersects $L$ transversely for all $t\in(0,\delta)$ and $s\in[0,1]$. This completes the proof of Step 1.\\
    
    \noindent\textbf{Step 2: identification of the boundary maps.}\par
    This is a direct consequence of Lemma \ref{lem:famlag} by letting $g_1^t=\mathcal{P}_{t,\epsilon}:L\rightarrow M$ with a given $\epsilon<\delta$.
    \end{proof}

\section{Calculation of examples on surfaces}
In this section, two examples are presented to illustrate the main theorem.

\begin{ex}
Consider the black curve $L$ in Figure \ref{fig_H1}. This is a Lagrangian immersion in the symplectic manifold $\R^2$. The red curve is obtained by performing a Hamiltonian flow to the right on the black curve. The morse function corresponding to this Hamiltonian function is the negative of the height function $h$. Figure \ref{fig_P1} illustrates $(\CF^P(L),\partial^P)$ defined by $h$. In Figure \ref{fig_H1}, the generators are the intersection points of the black and the red curve, marked in the diagram. The boundary map is calculated by counting bigons where the black curve forms the bottom boundary and the red curve forms the top boundary (see Figure \ref{fig_Holo_disc}). Specifically, the boundary map is described as follows, where generators that appear twice correspond to the same intersection points.

\begin{equation*}
        \xymatrix{
    d\ar[r]&j\ar[r]\ar[dr] & f\ar[r] & k\ar[r]\ar[dr] &a\\
    &c\ar[ur] \ar[dr]& b\ar[ur]& e\ar[ur]\ar[r]&d\\
    a\ar[r]&g\ar[r]\ar[ur]&h\ar[r]&e&
    }
\end{equation*}

\begin{figure}[H]
\centering 
\includegraphics[width=0.8\textwidth]{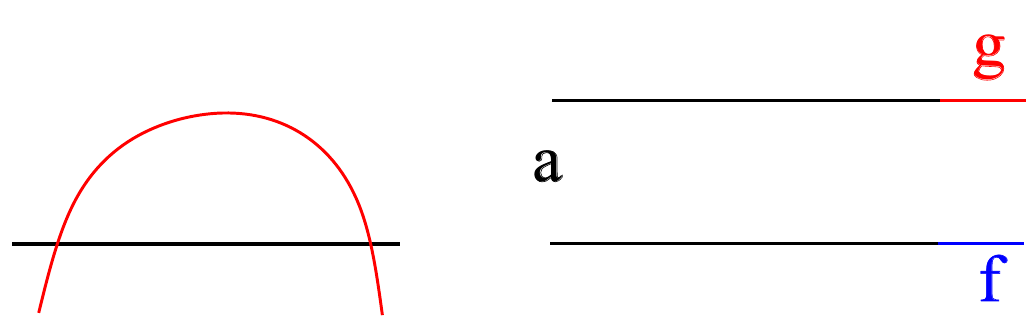}
\caption{The left picture is the holomorphic disc for Hamiltonian immersed Lagrangian Floer theory. The right picture is the disc part of the pearly tree from $a$ to $(g,f)$ (the morse trajectory part is constant).}
\label{fig_Holo_disc}
\end{figure}

For Figure \ref{fig_P1}, the generators are also marked in the diagram. The boundary map is defined by counting the pearly trees of Type 1 through 4, as specified in Definition \ref{defn:pearlydisc}. Note that pearly tree discs of Type 2 through 4 correspond to holomorphic discs, as stated in Corollary \ref{cor:nomorsetra}. The self-intersections of the black curve are branch switching points, with each self-intersection producing two generators. To distinguish these, the blue and red colors in the diagram indicate the order of the letters associated with those self-intersections.
\begin{itemize}
    \item If the blue branch is on the top when counting the holomorphic discs in the pearly trees, then the blue letter should appear first in the pair (see Figure \ref{fig_Holo_disc}).
    \item If the red branch is on the top when counting the holomorphic discs in the pearly trees, then the red letter should appear first in the pair.
\end{itemize}  
As a result, the boundary map of $(\CF^P(L),\partial^P)$ s described as follows, with generators that appear twice representing the same points.

\begin{equation*}
        \xymatrix{
    d\ar[r]&(j,h)\ar[r]\ar[dr] & (f,g)\ar[r] & k\ar[r]\ar[dr] &a\\
    &(c,b)\ar[ur] \ar[dr]& (b,c)\ar[ur]& e\ar[ur]\ar[r]&d\\
    a\ar[r]&(g,f)\ar[r]\ar[ur]&(h,j)\ar[r]&e&
    }
\end{equation*}

\begin{figure}[htbp]
\centering
\begin{minipage}[t]{0.4\textwidth}
\centering
\includegraphics[width=5cm]{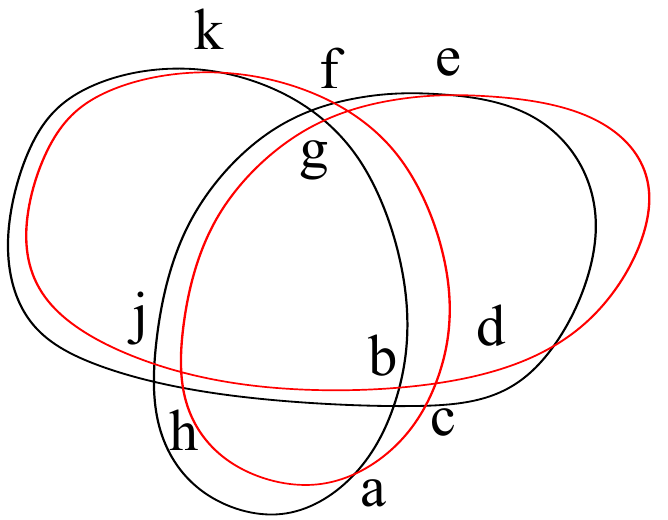}
\caption{Hamiltonian version}
\label{fig_H1}
\end{minipage}
\begin{minipage}[t]{0.4\textwidth}
\centering
\includegraphics[width=5cm]{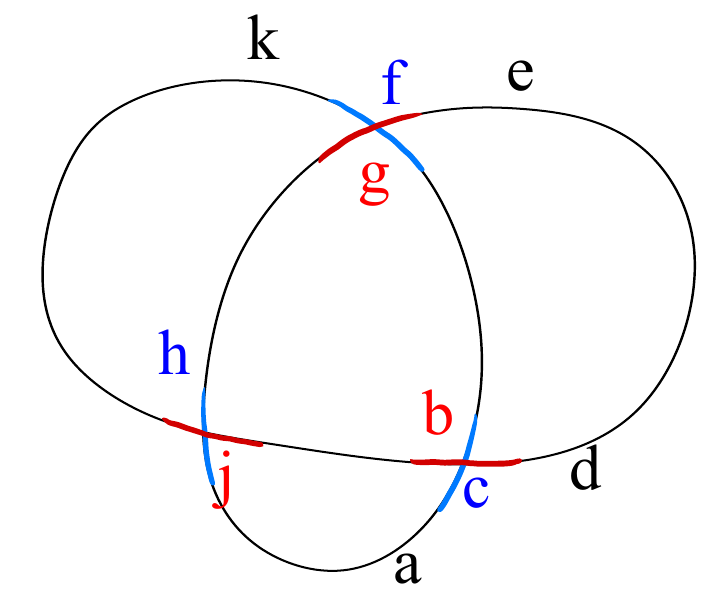}
\caption{Pearly tree version}
\label{fig_P1}
\end{minipage}
\end{figure}

\end{ex}

\begin{ex}

Consider the black curve $L$ in Figure \ref{fig_H2}. This is a Lagrangian immersion in the symplectic manifold $\R^2$. The red curve is given by performing a Hamiltonian flow to the right on the black curve. The Morse function corresponding to this Hamiltonian flow is the negative of the height function $h$. Figure \ref{fig_P2} represents $(\CF^P(L),\partial^P)$ defined by $h$. The same calculation method as in the previous example applies here, showing that the boundary map of $(\CF^P(L),\partial^P)$ for Figure \ref{fig_H2} is as follows:

\begin{equation*}
        \xymatrix{
    e\ar[r] & f\ar[r]\ar[dr] & i\ar[r] &h\\
    g\ar[r] & j\ar[r] \ar[ur]& a\ar[r]&b,
    }
\end{equation*}

The boundary map of Figure \ref{fig_P2} is:

\begin{equation*}
       \xymatrix{
    (e,d)\ar[r] & f\ar[r]\ar[dr] & i\ar[r] &(h,g)\\
    (g,h)\ar[r] & j\ar[r] \ar[ur]& a\ar[r]&(b,c),
    }
\end{equation*}

\begin{figure}[htbp]
\centering
\begin{minipage}[t]{0.4\textwidth}
\centering
\includegraphics[width=4cm]{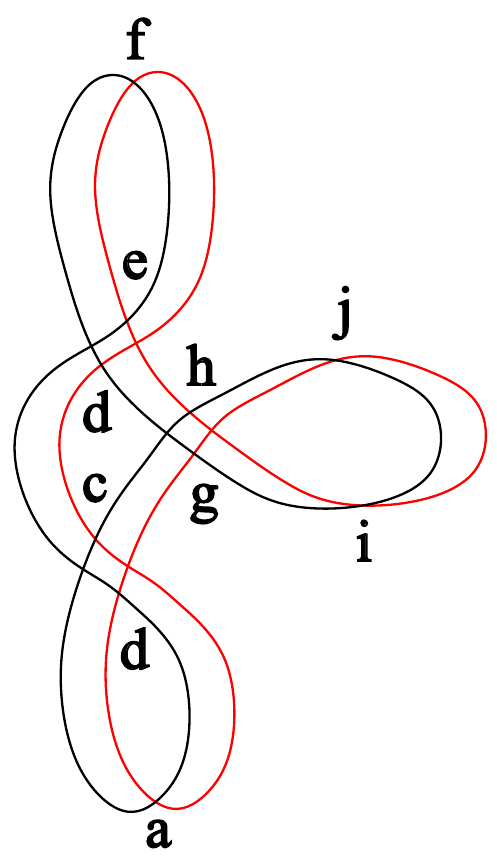}
\caption{Hamiltonian version}
\label{fig_H2}
\end{minipage}
\begin{minipage}[t]{0.4\textwidth}
\centering
\includegraphics[width=4cm]{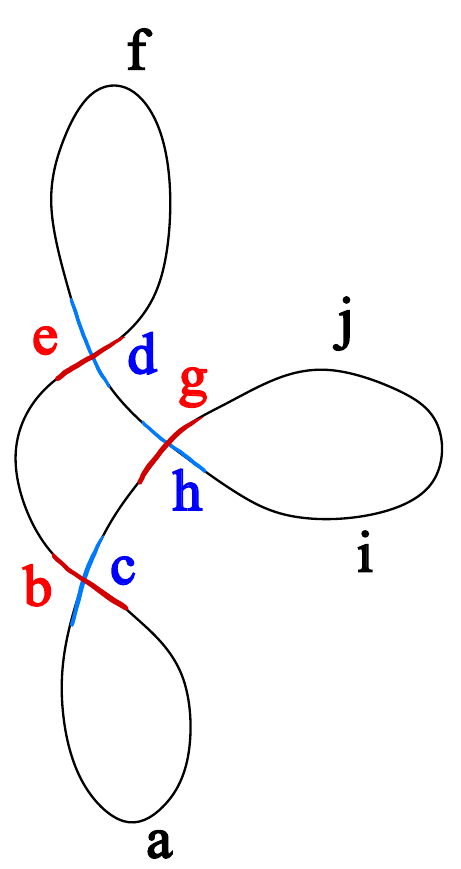}
\caption{Pearly tree version}
\label{fig_P2}
\end{minipage}
\end{figure}

\end{ex}

\bibliographystyle{plain}
\bibliography{references}

\end{document}